\documentclass[11pt,reqno]{amsart}
\usepackage{amssymb}
\usepackage[all]{xy}
\newtheorem{thm}{Theorem}[section]
\newtheorem{lem}[thm]{Lemma}
\newtheorem{prop}[thm]{Proposition}

\theoremstyle{definition}
\newtheorem{dfn}[thm]{Definition}

\theoremstyle{remark}
\newtheorem{remark}[thm]{Remark}

\newtheorem{notation}[thm]{Notation}

\newcommand{\CA}{{\mathcal{A}}}

\newcommand{\CE}{{\mathcal{E}}}

\newcommand{\CL}{{\mathcal{L}}}
\newcommand{\CB}{{\mathcal{B}}}

\newcommand{\af}{\alpha}
\newcommand{\bt}{\beta}
\newcommand{\gm}{\gamma}
\newcommand{\dt}{\delta}
\newcommand{\ep}{\epsilon}

\newcommand{\ld}{\lambda}

\newcommand{\sm}{\sigma}

\newcommand{\C}{{\mathbb{C}}}

\begin{document}


\title[Simple labeled graph $C^*$-algebras and disagreeable labeled spaces]
{Simple labeled graph $C^*$-algebras are associated to disagreeable labeled spaces}

\author[J. A. Jeong]{Ja A Jeong$^{\dagger}$}
\thanks{Research partially supported by NRF-2015R1C1A2A01052516$^{\dagger}$}
\thanks{Research partially supported by Hanshin University$^{\ddagger}$}
\address{
Department of Mathematical Sciences and Research Institute of Mathematics\\
Seoul National University\\
Seoul, 08826\\
Korea} \email{jajeong\-@\-snu.\-ac.\-kr }

\author[G. H. Park]{Gi Hyun Park$^{\ddagger}$}
\address{
Department of Financial Mathematics\\
Hanshin University\\
Osan, 18101\\
Korea} \email{ghpark\-@\-hs.\-ac.\-kr }




\begin{abstract} 
By a labeled graph $C^*$-algebra  we mean 
a $C^*$-algebra 
associated to a labeled space $(E,\mathcal L,\mathcal E)$ consisting of 
a labeled graph $(E,\mathcal L)$ and the smallest normal accommodating set 
$\mathcal E$ of vertex subsets.
Every graph $C^*$-algebra  $C^*(E)$ is a labeled graph $C^*$-algebra 
and it is well known that  $C^*(E)$ is simple 
if and only if the graph $E$ is cofinal and satisfies Condition (L). 
Bates and Pask extend these conditions of graphs $E$ to 
labeled spaces, and 
show  that if a set-finite and receiver set-finite 
labeled space $(E,\mathcal L, \mathcal E)$ is cofinal and disagreeable, 
then its $C^*$-algebra $C^*(E,\mathcal L, \mathcal E)$ is  simple. 
In this paper, we show that the converse is also true. 
\end{abstract}

\maketitle

\setcounter{equation}{0}

\section{Introduction}
  
A class of $C^*$-algebras associated to 
directed graphs including the Cuntz-Krieger algebras \cite{CK} 
was introduced in \cite{KPR, KPRR}, and 
since then its generalizations 
have attracted much attention of many authors. 
The $C^*$-algebras associated to ultragraphs, infinite matrices, 
higher-rank graphs, subshifts, 
Boolean dynamical systems, and labeled spaces are  
examples of the generalizations 
(see \cite{ BCP, BP1, BP2, Ca, DT, EL, Ma97, To1} among many others). 

 One of the main topics dealt with in the study of 
these generalized Cuntz-Krieger algebras is to describe 
the ideal structure of a $C^*$-algebra in question in terms of 
structural properties 
of the object to which the $C^*$-algebra  is associated.  
The ideal structure of a graph $C^*$-algebra is now well understood, and 
if we recall it for a row-finite graph $E$ with no singular vertices, 
it says that 
there exists a one to one 
correspondence between the gauge-invariant ideals of 
the  graph $C^*$-algebra $C^*(E)$ and 
the hereditary saturated vertex subsets of the graph $E$ 
(\cite{BPRS, BHRS, DT}), and 
moreover 
$C^*(E)$ is simple if and only if $E$ is cofinal and satisfies 
Condition (L) (\cite{DT,KPR}). 
Here the gauge action is the action of the unit circle on 
a graph $C^*$-algebra which always exists 
because of the universal property of a graph $C^*$-algebra. 
Many authors put a great deal of effort to extend this result 
to the classes of generalized Cuntz-Krieger algebras, and  
in this paper we will look at the labeled graph $C^*$-algebras 
and focus on the question of when these algebras are simple.   

If $(E,\CL)$ is a labeled graph, that is, 
$\CL:E^1\to \CA$ is a labeling map of the edges $E^1$ 
onto an alphabet $\CA$, then as we will review in 
the next section, one can consider 
a collection $\CB$ consisting of certain vertex subsets so that 
a universal family of projections $\{p_A: A\in \CB\}$ and 
partial isometries $\{s_a: a\in \CA\}$ 
satisfying the relations imposed by the triple $(E,\CL,\CB)$ 
exists and thus one can form the $C^*$-algebra $C^*(E,\CL,\CB)$ 
generated by this universal 
family of operators $\{p_A,s_a\}$. 
We call $C^*(E,\CL,\CB)$ the $C^*$-algebra of a labeled space $(E,\CL,\CB)$. 
Particularly, if $\CE$  is the smallest normal accommodating set,
we will simply call $C^*(E,\CL,\CE)$ 
the labeled graph  $C^*$-algebra of $(E,\CL)$ 
for convenience.  
In this paper, we will be mostly interested in 
these labeled graph $C^*$-algebras  $C^*(E,\CL,\CE)$. 

Every graph $C^*$-algebra is a labeled graph $C^*$-algebra 
(\cite[Example 5.1]{BCP}) and the class of Morita equivalence 
classes of  $C^*$-algebras of labeled spaces strictly 
contains the class of Morita equivalence  classes of graph 
$C^*$-algebras (see \cite[Remark 5.2]{BCP} and \cite[Theorem 3.7]{JKKP}).
By the universal property of a labeled graph 
$C^*$-algebra $C^*(E,\CL,\CB)$, there exists 
a gauge action of the unit circle 
on $C^*(E,\CL,\CB)$, and it is known \cite{JKiP} that 
if $E$ has no sinks and 
$(E,\CL,\CB)$ is a set-finite and receiver set-finite normal labeled space, 
there is a one to one correspondence between 
the gauge-invariant ideals of  $C^*(E,\CL,\CB)$ 
and the hereditary saturated  subsets of $\CB$. 
(A gauge invariant uniqueness theorem \cite[Theorem 5.3]{BP1} 
used in \cite{JKiP} was turned out to be incorrect, but 
was corrected in \cite[Theorem 2.7]{BPW} for 
normal labeled spaces and in \cite{BCP} for general labeled spaces.) 

Bates and Pask  \cite{BP2} 
considered the question of when 
a $C^*$-algebra $C^*(E,\CL,\CE)$ 
of a set-finite and receiver set-finite labeled space 
$(E,\CL,\CE)$ is simple, and proved that  
$C^*(E,\CL,\CE)$ is simple if $(E,\CL,\CE)$ is cofinal and disagreeable.
The notion of a disagreeable labeled space  $(E,\CL,\CE)$ 
introduced in \cite{BP2} is an analogue 
of Condition (L) of usual directed graphs. 
The  cofinal condition
for $(E,\CL,\CE)$ used in \cite{BP2} 
needs to be modified to obtain 
the simplicity result for $C^*(E,\CL,\CE)$  
as  noted in \cite[Remark 3.15]{JK} where
a condition called strongly cofinal was used instead.
The definition of a strongly cofinal labeled space given in 
\cite{JK} is weaker than the one in this paper 
(see Definition~\ref{strongly cofinal} or \cite[Section 2.5]{JKaP}), 
and throughout the present paper 
we mean  Definition~\ref{strongly cofinal} 
if we mention strong cofinality.
It then follows from \cite[Theorem 3.16]{JK} 
that $C^*(E,\CL,\CE)$ is simple whenever 
$(E,\CL,\CE)$ is disagreeable and strongly cofinal.

As for the converse of Bates and Pask's simplicity result, 
the strong cofinality of $(E,\CL,\CE)$ can be derived as in 
\cite[Theorem 3.8]{JK} by slightly modifying the proof there 
(see Theorem~\ref{thm-stcofinal}).
On the other hand, 
it is not clear whether the labeled space 
 $(E,\CL,\CE)$ has to be disagreeable when its $C^*$-algebra 
 $C^*(E,\CL,\CE)$ is simple, and 
 this is the question we will consider in this paper.
From a recent result \cite[Theorem 9,16]{COP} 
on simplicity of 
a $C^*$-algebra associated to a Boolean dynamical system, 
 we know that for a labeled space 
 $(E,\CL,\CE)$ whose Boolean dynamical system 
 satisfies a sort of domain condition, 
 the $C^*$-algebra $C^*(E,\CL,\CE)$ is simple if and only if 
$(E,\CL,\CE)$ has no cycles without an exit and there are no 
nonempty hereditary saturated subsets of $\CE$. 
The first condition of having no cycles without an exit is 
always satisfied whenever the labeled space is disagreeable 
while the converse does not hold in general 
(Proposition~\ref{prop_disagreeable}), 
and the second condition of having no nonempty 
saturated hereditary subsets is 
equivalent to the absence of 
gauge-invariant proper ideals in $C^*(E,\CL,\CE)$. 
Thus the question of  whether 
the converse of Bates and Pask's simplicity result 
holds true is not answered directly from 
\cite{COP} while it is known \cite[Theorem 3.14]{JK} 
that the converse holds if $\CE$ contains 
$\{v\}$ for every vertex $v$ in $E$.

The purpose of the present paper is, as mentioned above, 
to figure out whether the converse of Bates and Pask's simplicity result  
holds and it is proved in Theorem~\ref{main} that 
the labeled space $(E,\CL,\CE)$ is always disagreeable 
if $C^*(E,\CL,\CE)$ is simple.  
This establishes the following:  
a labeled graph $C^*$-algebra $C^*(E,\CL,\CE)$ is 
simple if and only if $(E,\CL,\CE)$ is strongly cofinal and 
disagreeable.  

\vskip 1pc 

\section{Preliminaries} 
In this section we set up notation and review 
definitions and basic results we need in the next section. 
For more details, we refer the reader to \cite{BCP} or \cite{JKaP}.

A  {\it directed graph} $E=(E^0,E^1,r,s)$
consists of  the vertex set $E^0$ and the edge set $E^1$ 
together with the range, source maps $r$, $s: E^1\to E^0$. 
We call  a vertex $v\in E^0$ a {\it sink} 
(a {\it source}, respectively)
if $s^{-1}(v)=\emptyset$  ($r^{-1}(v)=\emptyset$, respectively).
If every vertex in $E$ emits only finitely many edges, 
$E$ is called {\it row-finite}.  

For each $n\geq 1$, 
$E^n$  denotes the set of all  paths of length $n$, and 
the vertices in $E^0$ are regarded as finite paths of length zero. 
The maps $r,s$ naturally extend to the set 
$E^*=\cup_{n \geq 0} E^n$ of all finite paths, 
especially with $r(v)=s(v)=v$ for $v \in E^0$. 
By $E^{\infty}$ we denote the set of all infinite paths 
$x=\ld_{1}\ld_{2}\cdots$, where we define $s(x):= s(\ld_1)$. 
For $A,B\subset E^0$ and $n\geq 0$, we use the following notation 
 $$ AE^n: =\{\ld\in E^n :  s(\ld)\in A\},\ \
  E^nB: =\{\ld\in E^n : r(\ld)\in B\},$$
 and  $AE^nB: =AE^n\cap E^nB$ with  
$E^n v:=E^n\{v\}$, $vE^n:=\{v\}E^n$.
Also the sets of paths 
like $E^{\geq k}$, $AE^{\geq k}$, and $AE^\infty$ which 
have their obvious meaning will be used.  
A  {\it loop} is a finite path $\ld\in E^{\geq 1}$ 
such that $r(\ld)=s(\ld)$, and 
an {\it exit} of a loop $\ld$ is a path 
$\dt\in E^{\geq 1}$ such that  
$|\dt|\leq |\ld|,\ s(\dt)=s(\ld), \text{ and } 
\dt\neq \ld_1\cdots \ld_{|\dt|}.$ 
A graph $E$ is said to satisfy {\it Condition} (L)  
if every loop has an exit. 

Let $\CA$ be a countable alphabet and let  
$\CA^*$ ($\CA^\infty$, respectively) denote   
the set of all finite words (infinite words, respectively)
in symbols of $\CA$. 
A {\it labeled graph} $(E,\CL)$ over $\CA$ consists of 
a directed graph $E$ and  a {\it labeling map} 
$\CL:E^1\to \CA$ which is always assumed to be onto. 
Given a graph $E$, one can define a so-called 
{\it trivial labeling} map 
 $\CL_{id}:=id:E^1\to E^1$ which is the identity map 
 on $E^1$ with the alphabet $E^1$. 
The labeling map naturally extends to any 
finite and infinite labeled paths, namely 
if $\ld=\ld_1\cdots \ld_n\in E^n$, then  
 $\CL(\ld):=\CL(\ld_1)\cdots \CL(\ld_n)\in \CL(E^n)\subset \CA^*$, and  
 similarly to  infinite paths. 
 We often call these labeled paths just paths for convenience if 
 there is no risk of confusion,  
and use notation $\CL^*(E):=\CL(E^{\geq 1})$.
For a vertex $v\in E^0$ and a vertex subset $A\subset E^0$, 
we set $\CL(v): =v$ and $\CL(A): =A$, respectively. 
 A subpath  $\af_i\cdots \af_j$ of  
$\af=\af_1\af_2\cdots\af_{|\af|}\in \CL^*(E)$ 
 is denoted by
 $\af_{[i,j]}$ for $1\leq i\leq j\leq |\af|$, and 
 each $\af_{[1,j]}$, $1\leq j\leq |\af|$, is called an 
{\it initial path} of $\af$. 
 The range and source of a path $\af\in \CL^*(E)$ are defined 
to be the following sets of vertices
 \begin{align*}
r(\af) &=\{r(\ld) \in E^0 \,:\, \ld\in E^{\geq 1},\,\CL(\ld)=\af\},\\
 s(\af) &=\{s(\ld) \in E^0 \,:\, \ld\in E^{\geq 1},\, \CL(\ld)=\af\},
\end{align*}
and the {\it relative range of $\af\in \CL^*(E)$  
with respect to $A\subset  E^0$} is defined by
$$
 r(A,\af)=\{r(\ld)\,:\, \ld\in AE^{\geq 1},\ \CL(\ld)=\af \}.
$$
A collection  $\CB$ of subsets of $E^0$ is said to be
 {\it closed under relative ranges} for $(E,\CL)$ if 
$r(A,\af)\in \CB$ whenever 
 $A\in \CB$ and $\af\in \CL^*(E)$. 
We call $\CB$ an {\it accommodating set}~ for $(E,\CL)$
 if it is closed under relative ranges,
 finite intersections and unions and contains 
the ranges $r(\af)$ of all paths $\af\in \CL^*(E)$.
A set $A\in \CB$ is called {\it minimal} (in $\CB$)  
if $A \cap B$ is either $A$ or $\emptyset$ for all $B \in \CB$.
 
If $\CB$ is accommodating for $(E,\CL)$, 
the triple $(E,\CL,\CB)$ is called
 a {\it labeled space}. 
We say that a labeled space $(E,\CL,\CB)$ is {\it set-finite}
 ({\it receiver set-finite}, respectively) if 
for every $A\in \CB$ and $k\geq 1$ 
 the set  $\CL(AE^k)$ ($\CL(E^k A)$, respectively) is finite.
  A labeled space $(E,\CL,\CB)$ is said to be {\it weakly left-resolving} 
if 
 $$r(A,\af)\cap r(B,\af)=r(A\cap B,\af)$$
holds  for all $A,B\in \CB$ and  $\af\in \CL^*(E)$.
If $\CB$ is closed under relative complements, 
we call $(E,\CL, \CB)$ a {\it normal} labeled space.

\vskip 1pc 

\begin{notation}
For $A\in \CE$, we will use the following notation
$$A\sqcap \CB:=\{B\in \CE: B\subset A\}.$$ 
\end{notation}

\vskip 1pc 

\noindent 
{\bf Assumptions.}   
Throughout this paper, we assume that 
graphs $E$ have no sinks and sources,  
and labeled spaces $(E,\CL,\CB)$ are  
weakly left-resolving,  set-finite, receiver set-finite, and 
normal. 

\vskip 1pc 
 
\begin{dfn} 
\label{def-representation}
A {\it representation} of a labeled space $(E,\CL,\CB)$
is a family of projections $\{p_A\,:\, A\in \CB\}$ and
partial isometries
$\{s_a\,:\, a\in \CA\}$ such that for $A, B\in \CB$ and $a, b\in \CA$,
\begin{enumerate}
\item[(i)]  $p_{\emptyset}=0$, $p_{A\cap B}=p_Ap_B$, and
$p_{A\cup B}=p_A+p_B-p_{A\cap B}$,
\item[(ii)] $p_A s_a=s_a p_{r(A,a)}$,
\item[(iii)] $s_a^*s_a=p_{r(a)}$ and $s_a^* s_b=0$ unless $a=b$,
\item[(iv)]\label{CK4}  $p_A=\sum_{a\in \CL(AE^1)} s_a p_{r(A,a)}s_a^*.$ 
\end{enumerate}
\end{dfn}

\vskip 1pc  
\noindent
It is known \cite{BCP,BP1} that 
given a labeled space  $(E,\CL,\CB)$, 
there exists a $C^*$-algebra $C^*(E,\CL,\CB)$ generated by 
a universal representation $\{s_a,p_A\}$ of $(E,\CL,\CB)$, 
so that if $\{t_a, q_A\}$ is a representation of $(E,\CL,\CB)$ 
in a $C^*$-algebra $B$, there exists a $*$-homomorphism 
$$\phi: C^*(E,\CL,\CB)\to B$$ such that 
$\phi(s_a)=t_a$ and $\phi(p_A)=q_A$ for all $a\in \CA$ and 
$A\in \CB$. 

\vskip 1pc 
 
\begin{dfn} 
We call the $C^*$-algebra $C^*(E,\CL,\CB)$ generated by 
a universal representation of $(E,\CL,\CB)$ 
the {\it $C^*$-algebra of
a labeled space} $(E,\CL,\CB)$.
\end{dfn} 
 
\vskip 1pc
\noindent
The $C^*$-algebra $C^*(E,\CL,\CB)$ is unique up to isomorphism,  
and we simply write $$C^*(E,\CL,\CB)=C^*(s_a,p_A)$$ 
to indicate the generators $s_a, p_A$ that 
are nonzero for all $a\in \CA$ and  $A\in \CB$, $A\neq \emptyset$. 

\vskip 1pc

\begin{remark}\label{basics} 
Let $(E,\CL,\CB)$ be a labeled space with $C^*(E,\CL,\CB)=C^*(s_a,p_A)$. 
By $\ep$, we denote a symbol (not in $\CL^*(E)$)
such that  
$a\epsilon=\epsilon a$, $r(\ep)=E^0$, and $r(A,\ep)=A$ 
for all $a\in \CA$ and $A \subset E^0$. 
We write $\CL^{\#}(E)$ for the union $\CL^*(E)\cup \{\ep\}$.  
Let $s_\ep$ denote the unit of 
the multiplier algebra of $C^*(E,\CL,\CB)$.  
Then one can easily check the following 
$$(s_\af p_{A} s_\bt^*)(s_\gm p_{B} s_\dt^*)=
\left\{
   \begin{array}{ll}
      s_{\af\gm'}p_{r(A,\gm')\cap B} s_\dt^*, & \hbox{if\ } \gm=\bt\gm' \\
      s_{\af}p_{A\cap r(B,\bt')} s_{\dt\bt'}^*, & \hbox{if\ } \bt=\gm\bt'\\
      s_\af p_{A\cap B}s_\dt^*, & \hbox{if\ } \bt=\gm\\
      0, & \hbox{otherwise,}
   \end{array}
\right.
$$                    
for $\af,\bt,\gm,\dt\in  \CL^{\#}(E)$ and $A,B\in \CB$ 
(see \cite[Lemma 4.4]{BP1}). 
Since 
$s_\af p_A s_\bt^*\neq 0$ if and only if 
$A\cap r(\af)\cap r(\bt)\neq \emptyset$, 
we have
$$
C^*(E,\CL,\CB)=\overline{\rm span}\{s_\af p_A s_\bt^*\,:\,
\af,\,\bt\in  \CL^{\#}(E) ~\text{and}~ A \subseteq r(\af)\cap r(\bt)\}. 
$$
\end{remark}

\vskip 1pc

For a labeled graph $(E,\CL)$, there are many accommodating 
sets to be considered to form a labeled space, and 
the  $C^*$-algebras $C^*(E,\CL,\CB)$ 
 are not necessarily isomorphic to each other, in general.
By $\CE$  we denote the smallest accommodating set 
for which $(E,\CL, \CE)$ is a normal labeled space. 
We are mostly interested in the $C^*$-algebras of 
these labeled spaces 
$(E,\CL,\CE)$ throughout this paper.

For each $l\geq 1$,  
the relation $\sim_l$ on  $E^0$ given by 
$v\sim_l w$ if and only if 
$\CL(E^{\leq l} v)=\CL(E^{\leq l} w)$ 
is an equivalence relation, and the equivalence class 
$[v]_l$ of $v\in E^0$ is called a {\it generalized vertex} 
(or  a vertex simply).  
   If $k>l$,  then $[v]_k\subset [v]_l$ is obvious and
   $[v]_l=\cup_{i=1}^m [v_i]_{l+1}$
   for some vertices  $v_1, \dots, v_m\in [v]_l$ 
(\cite[Proposition 2.4]{BP2}). 
Moreover,  we have
$$
\CE=\Big\{ \cup_{i=1}^n [v_i]_l:\, v_i\in E^0,\  l\geq 1,\,  n\geq 0 \Big\},
$$
with the convention  $\sum_{i=1}^0 [v_i]_l:=\emptyset$ 
by \cite[Proposition 2.3]{JKK}. 

Recall that a {\it Cuntz-Krieger $E$-family} for a graph $E$ 
is a representation of the labeled space 
$(E,\CL_{id},\CE)$ with the trivial labeling, and 
{\it the Cuntz-Krieger uniqueness theorem} for graph $C^*$-algebras 
says that 
if $E$ satisfies Condition (L), 
then every Cuntz-Krieger $E$-family of nonzero 
operators generates the same $C^*$-algebra $C^*(E)$ 
up to isomorphism  
(for example, see \cite[Theorem 3.1]{BPRS}, \cite[Corollary 2.12]{DT}, 
and  \cite[Theorem 3.7]{KPR}).  
A condition of a  labeled space corresponding 
to Condition (L) of a directed graph was suggested in 
\cite{BP2} as below, and 
it is shown there in \cite[Lemma 5.3]{BP2} that 
for a graph $E$, 
the labeled space $(E,\CL_{id}, \CB)$ with the trivial labeling 
is disagreeable if and only if $E$ satisfies Condition (L).

\vskip 1pc 

\begin{dfn}\label{dfn-disagreeable} {\rm (\cite[Definition 5.2]{BP2})} 
 A path $\af \in \CL^*(E)$ with 
$s(\af)\cap [v]_l \neq \emptyset$ is 
called {\it agreeable} for $[v]_l$ if 
$\af=\bt\af'=\af'\gm$ for some $\af',\bt,\gm \in \CL^*(E)$ 
with $|\bt|=|\gm| \leq l$. 
Otherwise $\af$ is called {\it disagreeable}.
 We say that $[v]_l$ is {\it disagreeable} 
 if there is an $N\geq 1$ 
 such that for all $n > N$ there is an 
 $\af \in \CL(E^{ \geq n})$ which is disagreeable for $[v]_l$. 

A labeled space $(E,\CL, \CB)$ is said to be  {\it disagreeable}  
 if for every $v\in E^0$, there is an $L_v\geq 1$ such that 
  every $[v]_l$ is disagreeable for all $l \geq L_v$. 
\end{dfn} 
 
\vskip 1pc

It is then natural to ask whether every loop in a disagreeable labeled 
space must have an exit, which first leads us to 
try to seek a right definition for a loop in a labeled space 
and then to 
work on whether  the important results known for graph $C^*$-algebras 
$C^*(E)$ which involve the loop structure of $E$ 
can be generalized to labeled graph $C^*$-algebras. 
We take the following definition and will see  in 
the next section that 
every loop in a disagreeable labeled space has an exit.

\vskip 1pc 

\begin{dfn}\label{loop} {\rm (\cite[Definition 3.2]{JKK})}
Let $(E, \CL,  \CE)$ be a labeled space. 
For a  path $\af\in \CL^*(E)$ and  
a nonempty set $A \in  \CE$, 
we call  $(\af,A)$ a {\it loop} if  $$A\subset r(A,\af).$$  
We say that a loop  $(\af,A)$ has an {\it exit} 
if one of the following holds:
\begin{enumerate}
\item[(I)] there exists a path 
$\bt \in \CL(AE^{\geq 1})$ such that 
$|\bt|=|\af|,\ \bt \neq \af,$
\item[(II)] $A\subsetneq r(A,\af)$.
\end{enumerate} 
\end{dfn}

\vskip 1pc 

In \cite{COP}, 
the notion of cycle was introduced to define 
Condition $(L_\CB)$ for a labeled space $(E,\CL,\CB)$ 
(more generally for 
Boolean dynamical systems) which can be regarded as 
another condition 
analogous to Condition  (L) for usual directed graphs. 

\vskip 1pc 

\begin{dfn}
 (\cite[Definition 9.5]{COP}) 
For  $\af\in \CL^*(E)$ and a nonempty $A\in \CE$, 
the pair $(\af, A)$ is called a {\it cycle} if  
$$B=r(B,\af)\ \text{ for all  } B\in A\sqcap\CE.$$
Clearly every cycle is a loop, and 
if $(\af,A)$ is a cycle with an exit, then 
the exit must be of type (I). 
\end{dfn} 

\vskip 1pc

If $(E,\CL,\CE)$ is a labeled space satisfying our 
standing assumptions and if, in addition,
for each path $\af\in \CL^*(E)$, 
\begin{eqnarray} \label{domain}
r(D_\af, \af)=r(\af)\ \text{ for some } D_\af\in \CE
\end{eqnarray} 
(that is, every path has a domain in $\CE$), then 
the labeled graph $C^*$-algebra $C^*(E,\CL,\CE)$ 
can be regarded as a $C^*$-algebra associated to a 
Boolean dynamical system as discussed in 
\cite[Example 11.1]{COP}.  
A Boolean dynamical system on a 
Boolean algebra $\CB$ is said to satisfy 
Condition $(L_\CB)$ if it has no cycle  without an exit:  
if this is the case for the Boolean 
dynamical system induced from a labeled space 
$(E,\CL,\CE)$, we will simply say that 
$(E,\CL,\CE)$ satisfies Condition $(L_\CE)$.

Theorem~\ref{CK uniqueness thm} below 
is the {\it Cuntz-Krieger uniqueness theorem} 
for  labeled graph $C^*$-algebras: 
if  $(E,\CL,\CE)$ is disagreeable,   
it satisfies Condition  $(L_\CE)$ 
(see Proposition~\ref{prop_disagreeable} or 
 \cite[Proposition 3.7]{JKaP}). 
We need to understand Condition $(L_\CE)$ and 
disagreeability of labeled spaces  
not only to answer the simplicity question of 
labeled graph $C^*$-algebras, but also to be able to 
apply this useful 
uniqueness theorem for labeled graph $C^*$-algebras.

\vskip 1pc 

\begin{thm} {\rm (\cite[Theorem 5.5]{BP1}, \cite[Theorem 9.9]{COP})} 
\label{CK uniqueness thm} 
Let  $\{t_a, q_A\}$ be a representation of a labeled space $(E,\CL,\CE)$ 
such that $q_A\neq 0$ for all nonempty $A\in \CE$. 
If $(E,\CL,\CE)$ satisfies condition  $(L_\CE)$, in particular 
if  $(E,\CL,\CE)$ is disagreeable, then 
the canonical homomorphism $\phi:C^*(E,\CL,\CE)=C^*(s_a, p_A)\to 
 C^*(t_a, q_A)$ such that $\phi(s_a)=t_a$ and 
$\phi(p_A)=q_A$  is an isomorphism. 
\end{thm} 

\vskip 1pc
   
\begin{dfn} {\rm (\cite[Definition 3.4]{JKiP})} 
A subset  $H$ of $\CE$  is {\it hereditary} if 
it is closed under subsets, finite unions, and 
relative ranges in $\CE$.
A hereditary set $H$ is {\it saturated} if  $A \in H$ 
whenever $A \in \CE $ and $r(A,\af) \in H$ 
for all $\af \in \CL^*(E)$. 
\end{dfn}
 
\vskip 1pc

Let  $\overline{\CL(E^\infty)}$ 
be the set of all  infinite sequences 
$x\in \CA^{\mathbb N}$ such that 
every finite words of $x$ 
occurs as a labeled path in $(E,\CL)$, namely 
$$\overline{\CL(E^\infty)}:=\{x\in \CA^{\mathbb N}\mid 
x_{[1,n]}\in \CL(E^n) \ \text{for all } n\geq 1\}.$$ 
Clearly $\CL(E^\infty)\subset \overline{\CL(E^\infty)}$, and 
the notation $\overline{\CL(E^\infty)}$ comes from 
the fact that 
$\overline{\CL(E^\infty)}$ is the closure of $\CL(E^\infty)$ in the 
totally disconnected perfect space $\CA^{\mathbb N}$ 
which is equipped with the topology that has  a countable 
basis of open-closed cylinder sets 
$Z(\af):=\{x\in \CA^{\mathbb N}: x_{[1,n]}=\af\}$, $\af\in \CA^n$, $n\geq 1$ 
(see Section 7.2 of \cite{Ki}). 

\vskip 1pc
   
\begin{dfn}\label{strongly cofinal} {\rm (\cite[Section 2.5]{JKaP})} 
 We say that a labeled space $(E,\CL,\CE )$ is 
{\it strongly cofinal} if for each $x\in \overline{\CL(E^\infty)}$ 
and $[v]_l\in \CE $, there exist 
an $N\geq 1$ and a finite number of paths 
$\ld_1, \dots, \ld_m\in \CL^*(E)$ such that 
$$r(x_{[1,N]})\subset \cup_{i=1}^m r([v]_l,\ld_i).$$ 
\end{dfn}

\vskip 1pc
\noindent
The above definition of a strongly cofinal 
labeled space is stronger than the one 
given in \cite{JK}: for example, in the following labeled space 

\vskip 1pc 
\hskip .5cm
\xy  /r0.3pc/:(-44.2,0)*+{\cdots};(44.3,0)*+{\cdots ,};
(-40,0)*+{\bullet}="V-4";
(-30,0)*+{\bullet}="V-3";
(-20,0)*+{\bullet}="V-2";
(-10,0)*+{\bullet}="V-1"; (0,0)*+{\bullet}="V0";
(10,0)*+{\bullet}="V1"; (20,0)*+{\bullet}="V2";
(30,0)*+{\bullet}="V3";
(40,0)*+{\bullet}="V4";
 "V-4";"V-3"**\crv{(-40,0)&(-30,0)};
 ?>*\dir{>}\POS?(.5)*+!D{};
 "V-3";"V-2"**\crv{(-30,0)&(-20,0)};
 ?>*\dir{>}\POS?(.5)*+!D{};
 "V-2";"V-1"**\crv{(-20,0)&(-10,0)};
 ?>*\dir{>}\POS?(.5)*+!D{};
 "V-1";"V0"**\crv{(-10,0)&(0,0)};
 ?>*\dir{>}\POS?(.5)*+!D{};
 "V0";"V1"**\crv{(0,0)&(10,0)};
 ?>*\dir{>}\POS?(.5)*+!D{};
 "V1";"V2"**\crv{(10,0)&(20,0)};
 ?>*\dir{>}\POS?(.5)*+!D{};
 "V2";"V3"**\crv{(20,0)&(30,0)};
 ?>*\dir{>}\POS?(.5)*+!D{};
 "V3";"V4"**\crv{(30,0)&(40,0)};
 ?>*\dir{>}\POS?(.5)*+!D{};
 (-35,1.5)*+{a};(-25,1.5)*+{a};
 (-15,1.5)*+{a};(-5,1.5)*+{a};(5,1.5)*+{e_1};
 (15,1.5)*+{e_2};(25,1.5)*+{e_3};(35,1.5)*+{e_4};
 (0.1,-2.5)*+{v_0};(10.1,-2.5)*+{v_1};
 (-9.9,-2.5)*+{v_{-1}};
 (-19.9,-2.5)*+{v_{-2}};
 (-29.9,-2.5)*+{v_{-3}};
 (-39.9,-2.5)*+{v_{-4}}; 
 (20.1,-2.5)*+{v_{2}};
 (30.1,-2.5)*+{v_{3}};
 (40.1,-2.5)*+{v_{4}}; 
\endxy 
\vskip 1.5pc
\noindent
if
$x:=a^\infty\in  \overline{\CL(E^\infty)}\setminus \CL(E^\infty)$   
and $N,n,l\geq 1$, then  
$$r(x_{[1,N]})=r(a^N)=r(a)=\{v_{-k}: k\geq 0\}\nsubseteq 
\cup_{i=1}^m r([v_n]_l,\ld_i) $$ 
for any paths $\ld_1, \dots, \ld_m \in \CL^*(E)$, 
namely the labeled space is not strongly cofinal although 
it is in the sense of \cite{JK}.
The result \cite[Theorem 3.8]{JK} can be improved as below 
with a slightly modified proof which   
we provide here for the sake of readers' convenience.

\vskip 1pc
 
\begin{thm}\label{thm-stcofinal} 
If $C^*(E, \CL, \CE)$ is simple, then $(E, \CL, \CE)$  is strongly cofinal.
\end{thm}
\begin{proof}
Suppose to the contrary that there exist $[v]_l$ and  
$x\in  \overline{\CL(E^\infty)}$ 
such that
\begin{eqnarray}\label{assumption-stcofinal}
r(x_{[1,N]})\nsubseteq \cup_{i=1}^m r([v]_l,\ld_i)
\end{eqnarray}
for all $N\geq 1$ and any finite number of labeled paths 
$\ld_1, \dots, \ld_m$.
Let $I$ be the ideal generated by the projection $p_{[v]_l}$ and 
let $p_{x_1}:=p_{r(x_1)}$. 
Since $C^*(E,\CL,\CE)$ is simple, we must have $ p_{x_1}\in I$ and 
thus there is an element
$\sum_{j=1}^m c_j(s_{\af_j} p_{A_j} s_{\bt_j}^*)p_{[v]_l}
(s_{\gm_j} p_{B_j} s_{\dt_j}^*)$ in $I$ with $c_j\in \C$ 
such that
\begin{eqnarray}\label{approx}
\big\|\ \sum_{j=1}^m c_j(s_{\af_j} 
p_{A_j} s_{\bt_j}^*)p_{[v]_l}(s_{\gm_j} p_{B_j} s_{\dt_j}^*)
 -p_{x_1}\ \big\|< 1 
\end{eqnarray}
and the paths $\dt_j$'s have the same length 
$|\dt_j|=N_0\geq 1$.
Then
\begin{align*}
1>\ & \big\|\ \sum_{j}  
c_j(s_{\af_j} p_{A_j} s_{\bt_j}^*)p_{[v]_l}(s_{\gm_j} p_{B_j} s_{\dt_j}^*)
 -p_{x_1}\ \big\|\notag\\
\geq\ & \big\|\ \sum_{j} c_j(s_{\af_j} p_{A_j} s_{\bt_j}^*)p_{[v]_l}
(s_{\gm_j}p_{B_j} s_{\dt_j}^*) p_{x_1}
 -p_{x_1}\ \big\|\notag\\
=\ & \big\|\ \sum_{j} c_j(s_{\af_j} p_{A_j} s_{\bt_j}^*)p_{[v]_l}
(s_{\gm_j} p_{r([v]_l,\gm_j)\cap B_j \cap r(x_1\dt_j)}s_{\dt_j}^*)
 -p_{x_1}\ \big\|.
\end{align*}
We first show that for each $j=1,\dots, m$,
\begin{eqnarray}\label{delta-range}
r(x_1\dt_{j})\subset \cup_{i=1}^m r([v]_l, \gm_i).
\end{eqnarray}
If  $ r(x_1\dt_j) \nsubseteq \cup_{i=1}^m r([v]_l, \gm_i)$
for some $j$, then
  $r(x_1\dt_j)\setminus \cup_{i=1}^m r([v]_l, \gm_i)\neq \emptyset$ 
hence
  $ p_j:=p_{r(x_1\dt_j)\setminus \cup_{i=1}^m r([v]_l, \gm_i)}\neq 0.$ 
  Then with $J:=\{i\mid \dt_i=\dt_j\, \}$,
\begin{align*}
  1> &\ \big\|\ \big(\sum_{i} c_i(s_{\af_i} p_{A_i} s_{\bt_i}^*)p_{[v]_l}
(s_{\gm_i} p_{r([v]_l,\gm_i)\cap B_i \cap r(x_1\dt_i)}s_{\dt_i}^*)
 -p_{x_1}\big) s_{\dt_j}\ \big\|\\
 = & \ \big\|\  \sum_{i\in J} c_i(s_{\af_i} p_{A_i} s_{\bt_i}^*)p_{[v]_l}
 s_{\gm_i} p_{r([v]_l,\gm_i)\cap B_i \cap r(x_1\dt_i)}
 -p_{x_1}  s_{\dt_j}\ \big\|\\
 = & \ \big\|\  \sum_{i\in J} c_i(s_{\af_i} p_{A_i} s_{\bt_i}^*)p_{[v]_l}
 s_{\gm_i} p_{r([v]_l,\gm_i)\cap B_i \cap r(x_1\dt_i)}
 - s_{\dt_j}p_{r(x_1\dt_j)}  \ \big\|\\
  \geq & \ \big\|\  \sum_{i\in J} c_i(s_{\af_i} p_{A_i} s_{\bt_i}^*)p_{[v]_l}
 s_{\gm_i} p_{r([v]_l,\gm_i)\cap B_i \cap r(x_1\dt_i)}p_j
 - s_{\dt_j}p_{r(x_1\dt_j)} p_j \ \big\|\\
 = & \ \big\|\   s_{\dt_j} p_j \ \big\| =  1,
\end{align*}
which is a contradiction and (\ref{delta-range}) follows.
Also $\dt_i\neq x_{[2,N_0+1]}$ for each $1\leq i\leq m$.
In fact, if $\dt_i= x_{[2,N_0+1]}$ for some $i$,
then by (\ref{delta-range}),
$$r(x_{[1,N_0+1]})=r(x_1\dt_i) \subseteq \cup_{j=1}^m r([v]_l, \gm_j),$$
which is not possible because of (\ref{assumption-stcofinal}).
Thus $s_{\dt_i}^* s_{x_{[2,N_0+1]}} =0$ for
$i=1,\dots, m$.
 Then the partial isometry 
$y:=p_{x_1}s_{x_{[2,N_0+1]}}=s_{x_{[2,N_0+1]}}p_{r(x_{[1,N_0+1]})}$
 is nonzero since 
$s_{x_{[2,N_0+1]}}^*y=p_{r(x_{[1,N_0+1]})}\neq 0$, and
$s_{\dt_i}^*y=s_{\dt_i}^* p_{x_1}s_{x_{[2,N_0+1]}}
 =p_{r(x_1\dt_i)}s_{\dt_i}^*s_{x_{[2,N_0+1]}}=0$ 
for all $i$. 
 From (\ref{approx}), we have
$$
 1>  \ \big\|\ \big(\sum_{i=1}^m c_i(s_{\af_i}
 p_{A_i} s_{\bt_i}^*)p_{[v]_l}(s_{\gm_i}
p_{B_i} s_{\dt_i}^*)\big)yy^*  -p_{x_1}yy^*\ \big\| 
 =   \ \|yy^*\|=1,
$$
a contradiction, and we conclude that $(E,\CL,\CE)$ is strongly cofinal.
\end{proof}

\vskip 1pc
  
\section{Disagreeable labeled spaces of simple 
labeled graph $C^*$-algebras} 

In this section we prove that 
simplicity of a labeled graph $C^*$-algebra implies that 
the labeled space is disagreeable. 
For this, we will use condition $(c)$  
of the following lemma  
which is equivalent to disagreeability of a labeled space 
since the original definition of disagreeability 
seems a little complicated as 
recalled in  Definition~\ref{dfn-disagreeable}. 

\vskip 1pc 

\begin{lem} \label{lem_disagreeable} {\rm (\cite[Proposition 3.2]{JKaP})}
For a labeled space $(E,\CL,\CE)$, the following are equivalent:
\begin{enumerate}
\item[(a)] $(E,\CL,\CE)$ is disagreeable.  
\item[(b)] $[v]_l$ is disagreeable for all $v\in E^0$ and $l\geq 1$.
\item[(c)] For any nonempty set $A\in \CE$ and a path $\bt\in \CL^*(E)$, 
there is an $n\geq 1$ such that $\CL(AE^{|\bt|n})\neq \{\bt^n\}$.
\end{enumerate}
\end{lem}

\vskip 1pc 

If $(E,\CL,\CE)$ is disagreeable, then  
it satisfies Condition $(L_\CE)$ 
as shown in \cite[Proposition 3.7]{JKaP} and 
\cite[Example 11.1]{COP}. 
But the the converse is not true in general as we see from 
the following proposition. 

\vskip 1pc

\begin{prop}\label{prop_disagreeable} 
 Consider the following conditions of a labeled space 
$(E,\CL,\CE)$.  
\begin{enumerate}
\item[(a)]  $(E,\CL,\CE)$ is disagreeable. 
\item[(b)] Every loop in $(E,\CL,\CE)$ has an exit.
\item[(c)] $(E,\CL,\CE)$ satisfies $(L_\CE)$, that is, 
every cycle has an exit.  
\end{enumerate}
Then  $(a)\Rightarrow (b)\Rightarrow (c)$ hold. 
But the other implications are not true, 
in general.
\end{prop}
\begin{proof} 
$(b)\Rightarrow (c)$ 
is clear since every cycle is a loop. 

$(a)\Rightarrow (b)$  
Let $(A,\af)$ be a loop so that $A\subset r(A,\af)$.
If $A\subsetneq r(A,\af)$, then  the loop has an exit of type (II). 
So we may assume  that $A= r(A,\af)$. 
By Lemma~\ref{lem_disagreeable}, there is an $n\geq 1$ 
such that $\CL(AE^{n|\af|})\neq \{\af^n\}$. 
Choose $\bt\in \CL(AE^{n|\af|})$ with $\bt\neq \af^n$. 
If $\bt_{[1,|\af|]}\neq \af$, then 
$\CL(AE^{|\af|})\neq \{\af\}$ and 
$(A,\af)$ has an exit of type (I).  
If $\bt:=\af^k\dt$ for some $1\leq k\leq n-1$ and 
$\dt\in \CL^*(E)$ with $\dt_{[1,|\af|]}\neq \af$, then 
the loop $(A,\af)$ has an exit of type (I) since 
from $A=r(A,\af^k)$ we have 
$$\af\neq \dt_{[1,|\af|]}\in \CL(r(A,\af^k)E^{|\af|})=\CL(AE^{|\af|}).$$ 

$(b)\nRightarrow (a)$ 
The  labeled space $(E,\CL,\CE)$ of the 
following labeled graph is obviously not disagreeable   
while (b) is trivially satisfied  since it has no loops.

\vskip 1pc 
\hskip .5cm
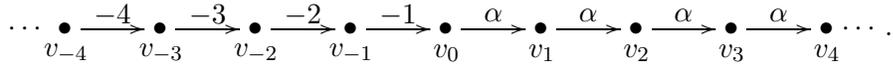
\xy  /r0.3pc/:(-44.2,0)*+{\cdots};(44.3,0)*+{\cdots .};
(-40,0)*+{\bullet}="V-4";
(-30,0)*+{\bullet}="V-3";
(-20,0)*+{\bullet}="V-2";
(-10,0)*+{\bullet}="V-1"; (0,0)*+{\bullet}="V0";
(10,0)*+{\bullet}="V1"; (20,0)*+{\bullet}="V2";
(30,0)*+{\bullet}="V3";
(40,0)*+{\bullet}="V4";
 "V-4";"V-3"**\crv{(-40,0)&(-30,0)};
 ?>*\dir{>}\POS?(.5)*+!D{};
 "V-3";"V-2"**\crv{(-30,0)&(-20,0)};
 ?>*\dir{>}\POS?(.5)*+!D{};
 "V-2";"V-1"**\crv{(-20,0)&(-10,0)};
 ?>*\dir{>}\POS?(.5)*+!D{};
 "V-1";"V0"**\crv{(-10,0)&(0,0)};
 ?>*\dir{>}\POS?(.5)*+!D{};
 "V0";"V1"**\crv{(0,0)&(10,0)};
 ?>*\dir{>}\POS?(.5)*+!D{};
 "V1";"V2"**\crv{(10,0)&(20,0)};
 ?>*\dir{>}\POS?(.5)*+!D{};
 "V2";"V3"**\crv{(20,0)&(30,0)};
 ?>*\dir{>}\POS?(.5)*+!D{};
 "V3";"V4"**\crv{(30,0)&(40,0)};
 ?>*\dir{>}\POS?(.5)*+!D{};
 (-35,1.5)*+{-4};(-25,1.5)*+{-3};
 (-15,1.5)*+{-2};(-5,1.5)*+{-1};(5,1.5)*+{\af};
 (15,1.5)*+{\af};(25,1.5)*+{\af};(35,1.5)*+{\af};
 (0.1,-2.5)*+{v_0};(10.1,-2.5)*+{v_1};
 (-9.9,-2.5)*+{v_{-1}};
 (-19.9,-2.5)*+{v_{-2}};
 (-29.9,-2.5)*+{v_{-3}};
 (-39.9,-2.5)*+{v_{-4}}; 
 (20.1,-2.5)*+{v_{2}};
 (30.1,-2.5)*+{v_{3}};
 (40.1,-2.5)*+{v_{4}}; 
\endxy 
\vskip 1.5pc

$(c)\nRightarrow (b)$ and $(c)\nRightarrow (a)$ 
Note that the labeled space $(E,\CL,\CE)$ 
of the following labeled graph, 
which is not disagreeable  clearly,  
has loops $(A_i,\af)$, $i=0,1$, where 
$A_0=\{v_0\}$ and $A_1:=r(\af)=\{v_0, v_1, \dots\}$. 
The loop $(A_1, \af)$ has no exits while $(A_0, \af)$ 
has an exit of type (II). Thus 
(b) is not satisfied for $(E,\CL,\CE)$. 
But the labeled space has no cycles and thus 
(c) is trivially satisfied.

\vskip 1pc 
\hskip .5cm
\xy  /r0.3pc/:(-44.2,0)*+{\cdots};(44.3,0)*+{\cdots .};
(-40,0)*+{\bullet}="V-4";
(-30,0)*+{\bullet}="V-3";
(-20,0)*+{\bullet}="V-2";
(-10,0)*+{\bullet}="V-1"; (0,0)*+{\bullet}="V0";
(10,0)*+{\bullet}="V1"; (20,0)*+{\bullet}="V2";
(30,0)*+{\bullet}="V3";
(40,0)*+{\bullet}="V4";
 "V-4";"V-3"**\crv{(-40,0)&(-30,0)};
 ?>*\dir{>}\POS?(.5)*+!D{};
 "V-3";"V-2"**\crv{(-30,0)&(-20,0)};
 ?>*\dir{>}\POS?(.5)*+!D{};
 "V-2";"V-1"**\crv{(-20,0)&(-10,0)};
 ?>*\dir{>}\POS?(.5)*+!D{};
 "V-1";"V0"**\crv{(-10,0)&(0,0)};
 ?>*\dir{>}\POS?(.5)*+!D{};
 "V0";"V1"**\crv{(0,0)&(10,0)};
 ?>*\dir{>}\POS?(.5)*+!D{};
 "V0";"V0"**\crv{(0,0)&(-5,4)&(0,8)&(5,4)&(0,0)};
 ?>*\dir{>}\POS?(.5)*+!D{};
 "V1";"V2"**\crv{(10,0)&(20,0)};
 ?>*\dir{>}\POS?(.5)*+!D{};
 "V2";"V3"**\crv{(20,0)&(30,0)};
 ?>*\dir{>}\POS?(.5)*+!D{};
 "V3";"V4"**\crv{(30,0)&(40,0)};
 ?>*\dir{>}\POS?(.5)*+!D{};
 (-35,1.5)*+{-4};(-25,1.5)*+{-3};
 (-15,1.5)*+{-2};(-5,1.5)*+{-1};(5,1.5)*+{\af};
 (0,8)*+{\af};
 (15,1.5)*+{\af};(25,1.5)*+{\af};(35,1.5)*+{\af};
 (0.1,-2.5)*+{v_0};(10.1,-2.5)*+{v_1};
 (-9.9,-2.5)*+{v_{-1}};
 (-19.9,-2.5)*+{v_{-2}};
 (-29.9,-2.5)*+{v_{-3}};
 (-39.9,-2.5)*+{v_{-4}}; 
 (20.1,-2.5)*+{v_{2}};
 (30.1,-2.5)*+{v_{3}};
 (40.1,-2.5)*+{v_{4}}; 
\endxy 
\end{proof}

\vskip 1.5pc 

\begin{lem}\label{minimal} {\rm (\cite[Lemma 4.6]{JKaP})} 
Let a labeled space $(E,\CL,\CE)$ have a loop $(A,\bt)$ 
without an exit. If $A$ is a minimal set in $\CE$, then 
the $C^*$-algebra $C^*(E,\CL,\CE)$ 
has a hereditary subalgebra which is isomorphic to 
$M_n(C(\mathbb T))$ for some $n\geq 1$,  
in particular $C^*(E,\CL,\CE)$ is not simple.
\end{lem}

\vskip 1pc

\begin{lem}\label{H_A} 
For a nonempty set $A\in \CE$, let  
$$H_{A} :=\big\{\cup_{i=1}^k\, C_i : k\geq 1,\  C_i\in 
 r(A,\bt)\sqcap \CE,\  \bt\in \CL^{\#}(E)\, \big\}.
$$
Then $H_A$ is a hereditary subset of $\CE$ and 
$$\bar H_{A}:=
\{ B\in \CE : \exists n\geq 1 \text{ such that } 
r(B,\af)\in H_{A} \text{ for all } \af\in \CL(E^{\geq n}) 
\, \}$$  
is a hereditary saturated subset of $\CE$ with $H_A\subset \bar H_{A}$. 
\end{lem}
\begin{proof}
It is easy to check that $H_{A}$ is a hereditary set. 
For convenience, we write a number $n$ in the definition of 
$\bar H_{A}$ for $B\in \CE$ as $n_B$ although it is not unique. 
Clearly $\bar H_{A}$ is  closed under  subsets.
If $A_1, A_2\in \bar H_{A}$, then 
$r(A_1\cup A_2,\af)=r(A_1,\af)\cup r(A_2,\af)\in  H_{A}$ 
whenever $|\af|\geq \max\{n_{A_1}, n_{A_2}\}$. 
Hence $\bar H_{A}$ is closed under finite unions. 
Let $B\in  \bar H_{A}$ and $|\sm|\geq 1$. 
Then $$r(r(B,\sm),\af)=r(B,\sm\af)\in H_{A}$$ whenever 
$|\af|\geq n_B$ because then $|\sm\af|\geq n_B$. 
Thus $\bar H_{A}$ is also closed under relative ranges, which  
shows that $\bar H_{A}$ is a hereditary subset of $\CE$. 
To see that $\bar H_{A}$ is saturated, let $B\in \CE$ 
satisfy $r(B,\af)\in \bar H_{A}$ 
for all paths $\af$ with $|\af|\geq 1$. 
We have to show that $B\in \bar H_{A}$. 
Since our labeled space is assumed to be set-finite, there are 
only finitely many labeled edges, say $\dt_1, 
\dots, \dt_k$, emitting out of $B$. 
Then  
$r(B,\dt_i)\in \bar H_{A}$ for each $i$,  and thus 
there is an $n_i\geq 1$  such that
$r(r(B,\dt_i),\af)\in H_{A}$ for all 
$\af\in \CL(E^{\geq n_i})$.
For $n:=\max_{1\leq i\leq k}\{n_i\}$, 
we then have 
$r(B,\dt_i\af)=r(r(B,\dt_i),\af)\in H_{A}$  
whenever $|\af| \geq n$ 
and $1\leq i\leq k$.
This means that $r(B, \af)\in H_{A}$ for all $\af$ 
with $|\af|\geq n+1$. 
Thus $B\in \bar H_{A}$ follows as desired. 
\end{proof}

\vskip 1pc

\begin{notation}
For a path $\bt:=\bt_1\cdots \bt_{|\bt|}\in \CL^*(E)$, 
let $$\bar{\bt}:=\bt\bt\bt\cdots$$ denote 
the infinite repetition of $\bt$, namely    
$$\bar{\bt}_1 \bar{\bt}_2 \bar{\bt}_3 \cdots 
=\bt_1\cdots\bt_{|\bt|}\bt_1\cdots \bt_{|\bt|}\cdots.$$ 
Then for each $j\geq 1$, 
we have $\bar{\bt}_j=\bt_k$ for some $1\leq k\leq |\bt|$
with   $k=j \, (\text{mod}\, |\bt|)$. 
The initial path $\bar{\bt}_1 \cdots \bar{\bt}_j$ 
 of $\bar\bt$, 
$j\geq 1$, is denoted by $\bar{\bt}_{[1,j]}$ as before.

\end{notation}

\vskip 1pc 

We will call a path $\bt\in \CL^*(E)$ {\it irreducible} 
if it is not a repetition of its proper initial path. 
The following Lemma~\ref{lem_N}   
will be used to derive a contradiction in 
the proof of Theorem~\ref{main}, but then  
we see from Theorem~\ref{main} that  
there does not exist a labeled space 
$(E,\CL,\CE)$ satisfying the assumptions of 
Lemma~\ref{lem_N}

\vskip 1pc

\begin{lem} \label{lem_N}
Let $C^*(E,\CL,\CE)$ be  a simple $C^*$-algebra and $A_0\in \CE$ 
be a nonempty set.   
If there exists  an irreducible path $\bt$ such that 
\begin{eqnarray}\label{beta} 
\CL(A_0 E^{\geq 1})
=\{\bar\bt_{[1,n]}:n\geq 1\}
=\{\bt^m\bt': m\geq 0, \ \bt' \text{ is an initial path of } \bt \}, 
\end{eqnarray}
or equivalently $\CL(A_0E^{n|\bt|})=\{\bt^n\}$ 
for all $n\geq 1$, then the following hold. 
\begin{enumerate}
\item[(i)] 
There is an $N\geq 1$ such that for all $k\geq 1$, 
$$r(A_0,\bar{\bt}_{[1,N+k]})
\subset \cup_{j=1}^N r(A_0,\bar{\bt}_{[1,j]}). $$ 

\item[(ii)] 
There is an $N_0\geq 1$ such that for all $k\geq 1$, 
\begin{eqnarray}\label{N_0} 
r(A_0,\bt^{N_0+k})\subset \cup_{j=1}^{N_0} r(A_0,\bt^j). 
\end{eqnarray}
Moreover $A=r(A,\bt)$ for $A:=\cup_{j=1}^{N_0} r(A_0,\bt^{j})$.
\end{enumerate}

\end{lem}

\begin{proof}  
We will frequently use  the following observation,
\begin{eqnarray}\label{j=k} 
 r(A_0,\bar\bt_{[1,j]})\cap r(A_0,\bar\bt_{[1,k]})\neq \emptyset 
\  \Rightarrow \  j=k\,(\text{mod}\, |\bt|).
\end{eqnarray} 
In fact, if  
$D:= r(A_0,\bar\bt_{[1,j]})\cap r(A_0,\bar\bt_{[1,k]})\neq \emptyset $ 
 for some  $j\neq k \,(\text{mod}\, |\bt|)$, then 
without loss of generality we can write  
$$  j:=m|\bt|+j_0\ \text{and }\  k:=n|\bt|+j_0+r$$ 
for some $m,n\geq 0$ and $0\leq j_0<|\bt|$, 
 $j_0<j_0+r< |\bt|$ 
(here we set $\bar\bt_{[1,0]}:=\epsilon$), 
then from (\ref{beta}) we must have
$$
\CL(DE^{|\bt|}) =\{\bt_{[j_0+1,j_0+r]}\bt_{[j_0+r+1,|\bt|]}\bt_{[1, j_0]}\} 
 =\{\bt_{[j_0+r+1,|\bt|]} \bt_{[1,j_0]}\bt_{[j_0+1,j_0+r]}\}. 
$$
But then 
the subpaths $$\mu:=\bt_{[j_0+1,j_0+r]}\text{ and }
 \nu:=\bt_{[j_0+r+1,|\bt|]}\bt_{[1, j_0]}$$ of $\bt$ 
satisfy $\mu\nu=\nu\mu$, which contradicts to irreducibility of $\bt$ 
(see \cite[Lemma 3.1]{JKaP}). 

(i) Since $\bar H_{A_0}$ is a nonempty hereditary saturated set  
by Lemma~\ref{H_A} and $C^*(E,\CL,\CE)$ is simple,  
it follows from \cite[Theorem 5.2]{JKiP} that  $\CE=\bar H_{A_0}$.  
Suppose 
\begin{eqnarray}\label{infinitely nonempty}  
r(A_0,\bar{\bt}_{[1,n]})\setminus 
\cup_{j=1}^{n-1} r(A_0, \bar{\bt}_{[1,j]})  \neq \emptyset 
\end{eqnarray}
for infinitely many $n\geq 1$.
Then by (\ref{j=k}), 
$r(A_0,\bt^n)\nsubseteq \cup_{j=1}^{n-1} r(A_0,\bt^j)$ 
for infinitely many $n$, which implies that 
$r(\bt^r)\notin H_{A_0}$ for all $r\geq 1$. 
In fact, if $r(\bt^r)=\cup_{i=1}^k C_i\in H_{A_0}$ with some 
$C_i\in r(A_0,\bar\bt_{[1,m_i]})\sqcap \CE$, then 
each $m_i$ must be a multiple of $|\bt|$ by  (\ref{j=k})  and 
thus for $m:=\max_i\{m_i\}/|\bt|$ we have 
$r(\bt^r)\subset \cup_{i=1}^{m} r(A_0,\bt^{i})$. 
But then for all sufficiently  large number $n >m|\bt|$, 
$$r(\bt^n)\subset r(\bt^r)\subset \cup_{i=1}^{m} r(A_0,\bt^{i})
\subset \cup_{j=1}^{n-1} r(A_0,\bar\bt_{[1,j]}),$$ 
which is not possible by (\ref{infinitely nonempty}).
Hence  $r(\bt^r)\notin H_{A_0}$ for all $r\geq 1$, which then  
easily implies that 
$r(\bt^r)\notin \bar H_{A_0}$ for all $r\geq 1$. 
But this contradicts to $\bar H_{A_0}=\CE$, and  
thus 
the left hand side of (\ref{infinitely nonempty}) must be empty 
for all  but finitely many $n$'s.
Therefore we see from (\ref{beta}) 
that there exists an $N\geq 1$ such that  
$$r(A_0,\bar{\bt}_{[1,n]})\subset 
\cup_{j=1}^{n-1} r(A_0, \bar{\bt}_{[1,j]})\ \text{ for all } 
n\geq N.$$
Then   
$r(A_0,\bar{\bt}_{[1,N+2]})
\subset \cup_{j=1}^{N+1} r(A_0,\bar{\bt}_{[1,j]})
\subset \cup_{j=1}^{N} r(A_0,\bar{\bt}_{[1,j]})$ because 
$r(A_0,\bar{\bt}_{[1,N+1]})\subset \cup_{j=1}^{N} r(A_0,\bar{\bt}_{[1,j]})$, 
and actually an induction gives
$$r(A_0,\bar{\bt}_{[1,N+k]})
\subset \cup_{j=1}^N r(A_0,\bar{\bt}_{[1,j]}) $$
for all $k\geq 1$, which proves (i).

(ii) 
We can take  $N=|\bt|N_0$, a multiple of $|\bt|$ in (i). 
Then $N_0$ satisfies (\ref{N_0}) by (i) and  (\ref{j=k}) since 
 (\ref{j=k}) implies that   for each $k\geq 1$, 
$$r(A_0,\bar{\bt}_{[1,N+k]})
\subset \underset{\substack{1\leq j\leq N\\ j=k ({\rm mod}\, |\bt|)}}{\cup} 
r(A_0,\bar{\bt}_{[1,j]}).$$ 
To show  $A=r(A,\bt)$ for 
$A:=\cup_{j=1}^{N_0} r(A_0,\bt^{j})$,
first note from  (\ref{N_0}) that  
$$A\supset r(A,\bt)\supset r(A,\bt^2)\supset \cdots.$$
Suppose $B:=A\setminus r(A,\bt)\neq \emptyset$. 
Then for $l>k\geq 1$, 
$$ r(B,\bt^k)\cap r(B,\bt^l) 
\subset  r(B,\bt^k)\cap r(A,\bt^{k+1}) 
=  r(B\cap r(A,\bt),\bt^k)=\emptyset.
$$
Thus 
$r(B,\bt^n)\setminus \cup_{j=1}^{n-1} r(B,\bt^j)\neq \emptyset$ 
for infinitely many $n$.
But this contradicts to  (\ref{N_0}) with $B$ in place of $A_0$ since 
$\CL(BE^{n|\bt|})=\{\bt^n\}$ for all $n\geq 1$, and 
we conclude that  $B=\emptyset$.
\end{proof}

\vskip 1pc 
 
In the following 
Theorem~\ref{main}, $(a)\Leftrightarrow (c)$  is  
known in \cite[Theorem 9.16]{COP} for the Boolean 
dynamical system  induced from a labeled space with  
the domain property (\ref{domain}). 
 
\vskip 1pc

\begin{thm}\label{main} 
Let $(E,\CL,\CE)$ be a labeled space. 
Then the following are equivalent: 
\begin{enumerate} 
\item[(a)] $C^*(E,\CL,\CE)$ is a simple $C^*$-algebra. 
\item[(b)]  $(E,\CL,\CE)$ is strongly cofinal and 
disagreeable.
 \end{enumerate}
Also these conditions imply the following. 
\begin{enumerate}
\item[(c)] $(E,\CL,\CE)$ has no cycles without 
exits and there is no proper 
hereditary saturated subsets in $\CE$.
\end{enumerate} 
If  $(E,\CL,\CE)$ satisfies the domain condition 
{\rm (\ref{domain})}, then 
$(c)$ is equivalent to $(a)$ and $(b)$.  

\end{thm}
\begin{proof} 
We only need to show that 
$(a)$ implies that $(E,\CL,\CE)$ is disagreeable. 

Suppose $(E,\CL,\CE)$ is not disagreeable. Then 
by Lemma~\ref{lem_disagreeable}, 
there exists a nonempty set $A_0\in \CE$ and a  
path $\bt\in \CL^*(E)$ 
such that for all $n\geq 1$, 
$$\CL(A_0 E^{|\bt|n})=\{\bt^n\},$$   
where we assume $\bt$  to be irreducible.   
Choose  an integer $N_0\geq 1$ such that  
$$r(A_0,\bt^{N_0+k})\subset \cup_{j=1}^{N_0} r(A_0,\bt^j) $$
for all $k\geq 1$, which exists by Lemma~\ref{lem_N}(ii). 
 Then for  
$$A:=\cup_{j=1}^{N_0} r(A_0,\bt^{j}),$$ 
 we have $A=r(A,\bt)$ 
by the same lemma. 
A simple computation shows that the hereditary subalgebra 
$p_A C^*(E,\CL,\CE) p_A$ of $C^*(E,\CL,\CE)$ 
generated by  
$p_A$  is equal to 
$$
Her(p_A):=\overline{\rm span}\{ s_{\mu}p_B s_\nu^*: 
\, B\in  r(A, \mu)\sqcap \CE,\ \mu,\nu\in  \bt^*_{[1,j]},\, 
0\leq j\leq |\bt|\,\},
$$
where we use notation 
$$\bt^*_{[1,j]}:=\{\bt^r\bt_{[1,j]}: r,j\geq 0 \} 
\ \text{ with } \bt^0:= \epsilon=:\bt_{[1,0]}.$$

Let $A_1\in A\sqcap \CE$ be a nonempty subset. 
Then 
 $\cup_{j=1}^{N} r(A_1,\bt^j)\subset A$ for all $N\geq 1$ 
since $A=r(A,\bt)$, but 
one can actually show that there exists an $N_1\geq 1$ such that 
\begin{eqnarray}\label{A_1}
A=\cup_{j=1}^{N_1} r(A_1,\bt^j).
\end{eqnarray}
In fact,  an integer $N_1\geq 1$ for which
$$\cup_{j=1}^{N_1} r(A_1,\bt^j)
=\cup_{j=1}^\infty r(A_1,\bt^j)$$ 
holds ($N_1$ exists again 
by Lemma~\ref{lem_N}(ii))  
satisfies (\ref{A_1}) because otherwise 
one can easily show that 
$$\emptyset\neq 
A\setminus \cup_{j=1}^{N_1} r(A_1,\bt^j)\ \notin 
\ \bar H_{A_1},$$ 
which is a contradiction to simplicity of $C^*(E,\CL,\CE)$ 
(or to $\bar H_{A_1}=\CE$). 

Now we claim that  $Her(p_{A_1})= Her(p_{A})$ 
for any nonempty subset $A_1\in A\sqcap \CE$. 
The hereditary subalgebra generated by $p_{A_1}$  is 
also equal to
$$Her(p_{A_1}) =\overline{\rm span}\{ s_{\mu}p_B s_\nu^*: 
\, B\in  r(A_1, \mu)\sqcap \CE,\ \mu,\nu\in  \bt^*_{[1,j]},\, 
0\leq j\leq |\bt|\,\},$$ 
and for each positive element of the form 
$s_{\mu}p_B s_\mu^*\in Her(p_A)$ 
with $B\in r(A,\mu)\sqcap \CE$, 
the following computation 
$$
s_{\mu}p_B s_\mu^* \leq s_{\mu}p_{r(A,\mu)} s_\mu^* 
= s_{\mu}p_{r\big(\cup_{i=1}^{N_1} r(A_1, \bt^i) ,\mu\big)} s_\mu^* 
\,\leq\, \sum_{i=1}^{N_1}  s_{\mu}p_{r(A_1, \bt^i\mu)} s_\mu^* 
$$
where we apply (\ref{A_1}) for the second equality shows 
that $s_{\mu}p_B s_\mu^*\in Her(p_{A_1})$. 
Then for each $s_{\mu}p_B s_\nu^*\in Her(p_A)$, 
the identity 
$$s_{\mu}p_B s_\nu^*
=(s_{\mu}p_B s_\mu^*)s_{\mu}p_B s_\nu^*(s_{\nu}p_B s_\nu^*)$$
proves that $s_{\mu}p_B s_\nu^*\in  Her(p_{A_1})$ 
(for this, see \cite[Theorem 3.2.2]{Mu}).
Thus $Her(p_{A_1})=Her(p_{A})$ follows for any nonempty 
subset $A_1\in A\sqcap \CE$. 
However, this is not possible if $A$ has a proper 
subset $A_1\in A\sqcap \CE$ since 
$p_A= p_{A_1}+p_{A\setminus A_1}\gneq p_{A_1}$. 
Hence $A$ must be a minimal set. 
But then, by Lemma~\ref{minimal} 
the $C^*$-algebra $C^*(E,\CL,\CE)$ 
contains a nonsimple hereditary subalgebra 
(isomorphic to $C(\mathbb T)$), and  
from this contradiction to simplicity of 
$C^*(E,\CL,\CE)$, we finally conclude that 
$(E,\CL,\CE)$ is disagreeable.
\end{proof}

\end{document}